\documentclass[12pt]{article}   
\usepackage{amsthm, amsfonts}
\usepackage{url}
\usepackage{amscd,amsmath,amssymb}

\newtheorem{definition}{Definition}
\newtheorem{theorem}{Theorem}
\newtheorem{lemma}{Lemma}
\newtheorem{proposition}{Proposition}

\def\de{\delta}

\def\al{\alpha}

\def\ga{\gamma}

\def\kappa{\varkappa}

\def\N{{\mathbb N}}
\def\M{{\cal M}}
\def\K{{\cal K}}
\def\GL{\operatorname{GL}}
\def\bm{\begin{pmatrix}}
\def\em{\end{pmatrix}}
\def\span{\operatorname{span}}
\def\rank{\operatorname{rank}}
\def\row{\operatorname{row}}

\def\beq{\begin{equation}}
\def\eeq{\end{equation}}
\def\bea{\begin{eqnarray*}}
\def\eea{\end{eqnarray*}}

\title{$\GL(n)$-dependence of matrices}

\author{N.~Tsilevich\thanks{Braude College of Engineering, Karmiel, Israel.
E-mail: {\tt natalia.tsilevich@gmail.com}.}
\and 
Y.~Manor\thanks{University of Haifa, Haifa, Israel. E-mail: {\tt yahel.manor49@gmail.com}.}
}

\author{N.~Tsilevich\thanks{Braude College of Engineering, Karmiel, Israel.
E-mail: {\tt natalia.tsilevich@gmail.com}.}
\and 
Y.~Manor\thanks{University of Haifa, Haifa, Israel. E-mail: {\tt yahel.manor49@gmail.com}.}
}

\date{}

\begin{document}
\maketitle

\begin{abstract}
We introduce the notion of $\GL(n)$-dependence of matrices, which is a generalization of  linear dependence taking into account the matrix structure. Then we prove a theorem, 
which generalizes, on the one hand, the fact that $n+1$ vectors in an $n$-dimensional vector space are linearly dependent and, on
the other hand, the fact that the natural action of the group $\GL(n,\K)$ on $\K^n\setminus\{0\}$ is transitive.
\end{abstract}

\section{Introduction}
There are various generalizations of the fundamental notion of linear (in)de\-pen\-dence of vectors, such as, for example, algebraic dependence in commutative algebra (see, e.g., \cite{Chamber}), the notion of matroid~\cite{Whitney} in combinatorics, forking~\cite{Shelah} in model theory, dominating in category theory~\cite{Isbell},   weak dependence~\cite{Hrbek}, $k$\nobreakdash-dependence~\cite{Feinberg}, etc. 
Recall that, given a vector space $V$ over a~field~$\K$, vectors $v_1,\ldots,v_k\in V$  are said to be linearly dependent if there are scalars $\al_1,\ldots,\al_k\in \cal K$ such that 
\beq\label{lindep}
\al_1v_1+\al_2v_2+\ldots+\al_kv_k=0, \quad \text{with not all $\al_i$'s zero}.
\eeq
If, for example, instead of vectors $v_i$'s we consider
elements $a_i$ of a field $\K$ and replace the linear equation~\eqref{lindep} by 
a polynomial equation with coefficients in  a~subfield $\cal L\subset \cal K$, we obtain 
the notion of algebraic dependence of $a_i$'s over~$\cal L$.

In this short note we generalize the notion of linear dependence and, correspondingly, Eq.~\eqref{lindep} in another direction, remaining fully in the framework of linear algebra. Namely, suppose that instead of vectors $v_i$'s we consider matrices $M_i$'s of the same order. Of course, the set of such matrices is a~linear space, so the ordinary definition of linear dependence applies. However, it does not take into account the matrix structure. In order to obtain a notion of dependence of matrices that does take into account the matrix structure, we replace multiplication by scalars with multiplication by matrices from the general linear group $\GL(n,\K)$. Namely, we introduce the following definition. 

For positive integers $n,m$, let $\M_{n\times m}$ be the set of $n\times m$ matrices over an arbitrary field $\K$, and let $\GL(n):=\GL(n,\K)$.

\begin{definition}\label{glndep}
Matrices $M_1,\ldots,M_k\in\M_{n\times m}$ are said to be {\rm $\GL(n)$-dependent} if there exist $g_1,\ldots, g_{m+1}\in\GL(n)\cup\{0\}$ such that
\beq\label{nado}
\sum\limits_{i=1}^{m+1}g_iM_i=0,\quad \text{with not all $g_i$'s zero}.
\eeq
\end{definition}

Observe that in the case of $n=1$ we obtain exactly the ordinary linear dependence of $m$-dimensional vectors (matrices of order $1\times m$).

Moreover, 
recall that any $m+1$ vectors in an $m$-dimensional space are linearly dependent. Our main result is the following theorem, which is a~direct analog of this claim for the case of $\GL(n)$\nobreakdash-dependence of matrices.

\begin{theorem}\label{maintheorem}
Any $m+1$ matrices from $\M_{n\times m}$ are $\GL(n)$-dependent.
\end{theorem}

The $n=1$ case of the theorem is exactly the claim mentioned above, while the $m=1$ case is essentially the transitivity of the action of $\GL(n)$ on~$\K^n\setminus\{0\}$ (see Lemma~\ref{lem:transitivity}). Thus, the theorem is a simultaneous generalization of both these facts.

Although the result is purely linear algebraic, the original motivation for it came from computer science theory. Namely,
a major open problem in circuit complexity is to prove lower bounds
on the depth of circuits solving problems in $P$ beyond $O\left(\log n\right)$, and 
one of the approaches to this problem is via the so-called KRW conjecture, which
claims that the circuit depth complexity of functions
of the form $f(g(i_{1}),g(i_{2}),\dots,g(i_{n}))$
is at least the sum of their depth complexities minus some small loss. 
Theorem 1 appeared as a conjecture in a joint project of the second author and O. Meir dealing with a simplified version of the KRW conjecture (known as the semi-montone composition), where it was needed for the proof of parity query complexity analog (For more details, see~\cite{ManorMeir}.)

The proof of the theorem for a finite field (Sec.~\ref{sec:finite}) is quite easy, while in the case of an infinite field, we need a more involved argument
 (Sec.~\ref{sec:infinite}). In Sec.~\ref{sec:lindep} we essentially restate Definition~\ref{glndep} and Theorem~\ref{maintheorem} in terms of linear subspaces instead of matrices.

\section{Proof of the theorem for finite fields}\label{sec:finite}

Let $\mathcal{K}$  be a {\it finite} field. In
order to prove the main theorem over $\K$, we need the following result.

\begin{lemma}[{\cite[Theorem 4]{GJJR10}}]\label{prop:aux} There exists a linear
subspace $H\subset\mathcal{M}_{n\times n}$ such that $\dim H=n$ and
every nonzero matrix $M\in H$ is of full rank.
\end{lemma}

In other words, $H\subset\GL(n)\cup\{0\}$.

\begin{proposition}\label{prop:finite}
Theorem~\ref{maintheorem} holds for a finite field $\K$.
\end{proposition}

\begin{proof}
Let $H$ be the the linear subspace from Proposition~\ref{prop:aux}. Given matrices $M_1,\ldots, M_{m+1}\in\M_{n\times m}$, consider the linear function 
$f\colon H^{m+1}\to\M_{n\times m}$ defined by
\[
f\left(g_{1},\dots, g_{m+1}\right)=\sum_{i=1}^{m+1}g_{i}M_{i},\qquad g_i\in H.
\]
Denoting by $\text{dom}f$ and $\text{img}f$ the domain and the image of $f$, respectively, it is easy to see that
\[
\dim\text{dom}f=\left(m+1\right)n> mn\ge\dim\text{img}f.
\]
Hence, there exists a nonzero assignment $(g_{1},\dots, g_{m+1})\in H^{m+1}$ such
that $f(g_{1},\dots,g_{m+1})=0$. 
Since $H\subset\GL(n)\cup\{0\}$, this completes the proof. 
\end{proof}

\section{Proof of the theorem for infinite fields}\label{sec:infinite}

The aim of this section is to prove the main theorem in the more difficult case of an infinite field. We need the following easy lemmas.

\begin{lemma}
Let $v_1,\ldots,v_{k}$ be vectors in an arbitrary linear space such that there exist scalars $\al_1,\ldots,\al_k$, not all zeros, such that 
$\sum\limits_{i=1}^{k}\al_iv_i=0$. 
For a~fixed~$j$, if  $\al_j\ne0$ for every expansion of this form, then the vectors $v_i$, $i\ne j$, are linearly independent.
\label{lemma}
\end{lemma}

\begin{proof}
Immediately follows from the definition of linear (in)dependence.
\end{proof}

\begin{lemma}\label{lem:transitivity}
For every $w_1,w_2\in\M_{n\times 1}$ there exist $g_1,g_2\in\GL(n)\cup\{0\}$ such that at least one of them is nonzero and $g_1w_1+g_2w_2=0$.
\label{trans}
\end{lemma}

\begin{proof}
If one of the vectors, say $w_1$, is zero, then we take $g_2=0$ and $g_1$ an arbitrary matrix from $\GL(n)$. Further, it is well known that the action of $\GL(n)$ on $\K^n\setminus\{0\}$ is transitive, so,  if $w_1,w_2\ne0$, then there exists $g\in\GL(n)$ such that $gw_1=w_2$, and we take $g_1=g$ and $g_2=-I$ where $I$ is the identity matrix.
\end{proof}

Note that this lemma is the special case $m=1$ of the main theorem. Now we are ready to prove it for infinite fields.

\begin{theorem}
Theorem~\ref{maintheorem} holds for an infinite field $\K$.
\end{theorem}

\smallskip\noindent{\bf Remark.} Obviously, the theorem remains valid if the number of matrices is greater than $m+1$ (we can simply take $g_i=0$ for ``superfluous'' matrices).

\begin{proof}
Observe that if $M_j=0$ for some $j$, then we can take $g_j=I$ and $g_i=0$ for $i\ne j$ to obtain~\eqref{nado}. So, {\it in what follows we assume that $M_i\ne0$ for all $i$.}

The proof proceeds by double induction, the outer induction on $n$ and the inner one on $m$.

As mentioned in the introduction, the {\bf base case  $\mathbf{n=1}$ of the outer induction} is exactly the fundamental theorem that $m+1$ vectors in an $m$-dimensional space are  linearly dependent.

\smallskip
{\bf Induction step of the outer induction} proceeds by  induction on~$m$. 

{\bf Base case $\mathbf{m=1}$ of the inner induction}  is covered by Lemma~\ref{trans}. 

{\bf Induction step of the inner induction}.
To make the idea of the proof clear,  we 
 first consider the {\bf case $\mathbf{n=2}$}. The proof of the general case essentially repeats the same argument, but with more complicated notation.
 
 For $n=2$ we have $w_1,\ldots,w_n\in\M_{2\times m}$, i.e., $w_i=\bm u_i \\ v_i\em$ with $u_i,v_i\in\K^m$. 

By the $n=1$ case, there exist $a_1,\ldots,a_{m+1}\in\K$  and $b_1,\ldots,b_{m+1}\in\K$ such that 
\begin{eqnarray}
\sum\limits_{i=1}^{m+1}a_iu_i&=&0,\qquad\text{with not all $a_i$'s zero},\label{a}\\
\sum\limits_{i=1}^{m+1}b_iv_i&=&0,\qquad\text{with not all $b_i$'s zero}.\label{b} 
\end{eqnarray}
For each $i\in[m+1]$, consider the matrix $g_i=\bm a_i & 0\\0&b_i\em$. Then 
$$
\sum\limits_{i=1}^{m+1}g_iw_i=\sum\limits_{i=1}^{m+1}\bm a_i & 0\\0&b_i\em\bm u_i \\ v_i\em=\bm\sum\limits_{i=1}^{m+1}a_iu_i\\
\sum\limits_{i=1}^{m+1}b_iv_i\em=0.
$$ 
Let us say that an index $j$ is {\it bad} if $\det g_j=0$, and {\it good} otherwise. If all indices are good,  
then we are done.

\smallskip
{\bf Step 1}: Assume that there is an index $j$ such that  
\beq
u_j\notin \span\{u_i,v_i\}_{i\ne j}.\label{0}
 \eeq
Consider the vector $v_j$. 
If there is no expansion~\eqref{b} with $b_j=0$, then,
by Lemma~\ref{lemma}, the vectors $v_i$, $i\ne j$, are linearly independent. But there are $m$ of them, so  $\span\{v_i\}_{i\ne j}=\K^m$,  a contradiction with~\eqref{0}.
Therefore, we can find an expansion~\eqref{b} with $b_j=0$. Then we have $g_j=0$. 

Denote $V=\span\{u_i,v_i\}_{i\ne j}$.
It follows from~\eqref{0} that $r:=\dim V\le m-1$. 
Let $\phi:V\to\K^r$ be an isomorphism of linear spaces. Denote by~$\phi^2$ the corresponding isomorphism $V^2\to\M_{2\times r}$, i.e., $\phi^2\big(\bm u \\ v\em\big)=\bm \phi(u) \\ \phi(v)\em$. Observe that $\phi^2$ commutes with every $g\in\GL(2)$.
Now,
denoting $w'_i=\phi^2(w_i)$, we see that $\{w'_i\}_{i\ne j}$ is a family of $m\ge r+1$ matrices from $\M_{2\times r}$. By the induction hypothesis, there exist  $h_i\in\GL(2)\cup\{0\}$, $i\ne j$, with not all $h_i$'s zero, such that $0=\sum_{i\ne j}h_iw'_i=\sum_{i\ne j}h_i\phi^2(w_i)=\phi^2(\sum_{i\ne j}h_i w_i)$, which implies that $\sum_{i\ne j}h_i w_i=0$. So, taking $g_i=h_i$ for $i\ne j$ and $g_j=0$, we obtain~\eqref{nado}.

In the same way we treat the case where there exists $j$ such that $v_j\notin \span\{u_i,v_i\}_{i\ne j}$.

\smallskip
Thus, from now on we assume that $u_j,v_j\in\span\{u_i,v_i\}_{i\ne j}$ for all $j$.

\smallskip
{\bf Step 2.}
Now we will successively consider all bad indices, at each step 
``correcting'' the current matrices $g_i=\bm a_i& c_i\\ d_i &b_i\em$ so that (i) the equation $\sum g_iw_i=0$ is preserved; (ii) if $i$ was good, then it remains good.

Let $j$ be a bad index (i.e., $\det g_j=0$).
 Recall that $u_j,v_j\in \span\{u_i,v_i\}_{i\ne j}$, i.e.,
$u_j=\sum\limits_{i\ne j}(\al_iu_i+\ga_iv_i)$, $v_j=\sum\limits_{i\ne j}(\de_iu_i+\beta_iv_i)$ for some scalars $\al_i,\beta_i,\gamma_i,\delta_i$. 

Now change the $g_i$'s as follows (a nonzero constant $ x\in\K$ is to be chosen later):
$$
g'_j:=g_j+x\bm 1 & 0\\0&1\em,\qquad g'_i:=g_i-x \bm \al_i& \ga_i\\ \de_i&\beta_i\em\quad\text{for $i\ne j$}.
$$
Then
\bea
\sum_{i=1}^{m+1} g'_iw_i&=&\sum_{i=1}^{m+1} g_iw_i+\bm x & 0\\0&x\em\bm u_j\\v_j\em+\sum_{i\ne j} \bm -x\al_i& -x\ga_i\\-x\de_i&-x\beta_i\em\bm u_i\\v_i\em\\
&=&0+\bm xu_j-x\sum_{i\ne j}(\al_iu_i+\ga_iv_i)\\ xv_j-x\sum_{i\ne j}(\de_iu_i+\beta_iv_i)\em=0.
\eea

We want to ensure that (a) $\det g'_j\ne0$; (b)
 for $i\ne j$, if $\det g_i\ne0$ then $\det g'_i\ne0$.
 
But $\det g'_j=\det g_j+x(a_j+b_j)+x^2=x(a_j+b_j)+x^2$, so the condition $\det g'_j\ne0$ forbids at most two values for $x$. 

Further, $\det g'_i=\det g_i-x(\al_ib_i+\beta_ia_i-\ga_id_i-\de_ic_i)+x^2(\al_i\beta_i-\ga_i\de_i)$ for $i\ne j$, so if $\det g_i\ne0$, then the condition  $\det g'_i\ne0$ also forbids at most two values for $x$.
 
Therefore, the field being infinite, we can find $x$ as required.

After each step of this procedure, we
denote $g'_i$ again by $g_i$ (note that conditions (i)--(ii) are satisfied, and the number of  bad indices has decreased by one) and proceed to the next bad index. Thus, successively applying this procedure to all bad indices, we obtain an expansion of the required form with $g_i\in\GL(2)$ for all $i$, which completes the proof of the case $n=2$. 

\smallskip
{\bf General case}. 
Now we have $M_1,\ldots,M_{m+1}\in\M_{n\times m}$, i.e., $M_i=\bm u_i^{(1)}\\ \ldots\\ u_i^{(n)}\em$ with
$u_i^{(k)}\in\K^m$. 

By the $n=1$ case, there exist  $a_i^{(k)}\in\K$, $i=1,\ldots, m+1$, $k=1,\ldots, n$, such that for each $k$
\beq
\sum\limits_{i=1}^{m+1}a_i^{(k)}u_i^{(k)}=0,\qquad\text{with not all $a_i^{(k)}$'s zero}.\label{aa} 
\eeq

Let $g_i$ be the $n\times n$ diagonal matrix with diagonal entries $a_i^{(k)}$. Then $\sum g_iM_i=0$. 
We say that an index $j$ is {\it bad} if $\det g_j=0$, and {\it good} otherwise.
If all indices are good, then we are done.  

\smallskip
Assume that there is an index $j$ such that: 
\beq
\exists\ell\in[n] \text{ such that }u^{(\ell)}_j\notin\span\{u^{(k)}_i\}_{k\in[n],\,i\ne j}.\label{*}
\eeq
Then, exactly as at Step~1 above, for each $k\ne\ell$ we  find an expansion~\eqref{aa} with $a_j^{(k)}=0$ and
obtain $g_j=0$, and then apply the induction hypothesis.

\smallskip
So, from now on we assume that for all  indices $j$  we have 
\beq
u_j^{(\ell)}\in\span\{u^{(k)}_i\}_{k\in[n],\,i\ne j}\quad\text{ for all }\ell\in[n].\label{***}
\eeq

Now, as at Step~2 above, we will successively consider all bad indices, at each step 
``correcting'' the current matrices $g_i=(a^{(i)})_{r,s=1}^n$ so that (i) the equation $\sum g_iM_i=0$ is preserved; (ii) if an index was good, it remains good.

So, let $j$ be a bad index. 
For each $\ell\in[n]$, by~\eqref{***} we have
$u_j^{(\ell)}\in\span\{u^{(k)}_i\}_{k\in[n],\,i\ne j}$, that is, $u_j^{(\ell)}=\sum_{k\in[n],\,i\ne j}\al_{ik}^{(\ell)}u^{(k)}_i$ for some scalars $\al_{ik}^{(\ell)}$.
Change the current $g_i$'s as follows (a~nonzero constant $x$ is to be chosen later):
$$
g'_j:=g_j+xI,\qquad g'_i:=g_i-x\sum_{\ell=1}^n\sum_{k=1}^n\al_{ik}^{(\ell)}E_{\ell k}\quad\text{ for } i\ne j,
$$
where $I$ is the $n\times n$ identity matrix and $E_{rs}$ is a matrix unit. Then it is easy to see that 
$\sum g'_iM_i=0$.

We want to ensure that (a) $\det g'_j\ne0$; (b)
 for $i\ne j$, if $\det g_i\ne0$ then $\det g'_i\ne0$.
 
Condition~(a) has the form $P(x)\ne0$ where $P$ is a polynomial in $x$ with leading term $x^{n}$, while condition~(b) for a fixed $i$ has the form $P(x)\ne0$ where $P(x)$ is a polynomial in $x$ of degree at most $n$ with free term $\det g_i\ne0$. So,
each of the conditions forbids at most $n$ values for $x$ and, therefore, the field being infinite, we can find a suitable $x$. 

After each step of this procedure, we
denote $g'_i$ again by $g_i$ (note that conditions (i)--(ii) are satisfied, and the number of  bad indices has decreased by one) and proceed to the next bad index. Thus, successively applying this procedure to all bad indices, we obtain an expansion of the required form with $g_i\in\GL(n)$ for all $i$, which completes the proof. 
\end{proof}

Thus, Theorem~\ref{maintheorem} is proved in full generality.

\section{$\GL(n)$-dependence of subspaces}\label{sec:lindep}

In this section we restate our definition and the main theorem in terms of subspaces.

Given a matrix $M\in\M_{n\times m}$, denote by $\row M$ its row space, i.e., the subspace in $\K^m$ spanned by the rows of $M$.  The following lemma is well known.

\begin{lemma}\label{class}
Let $M_1,M_2\in\M_{n\times m}$. Then 
$M_2=gM_1$ where $g\in\GL(n)$ if and only if $\row M_1=\row M_2$.
\end{lemma}

Therefore, a $\GL(n)$-orbit in $\M_{n\times m}$ is determined by a linear subspace $L\subset\K^m$ and consists of all matrices $M$ with $\row M=L$. This suggests the following definition.

\begin{definition}
Subspaces $L_1,\ldots,L_k\subset\K^m$ are said to be {\rm $\GL(n)$-dependent} if there exist $x^{(i)}_j\in L_i$ 
for $i=1,\ldots, k$, $j=1,\ldots n$,  such that
\begin{itemize}
\item[\rm(a)] 
$$
\sum_{i=1}^k x^{(i)}_j=0,\quad j=1,\ldots,n;
$$
\item[\rm(b)] 
$$
\span\{x^{(i)}_j\}_{j=1}^n\text{ is either }L_i\text{ or } \{0\}, \quad{i=1,\ldots,k}, \text { not all }\{0\}.
$$
\end{itemize}
\end{definition}

Thus, given matrices $M_1,\ldots,M_k\in\M_{n\times m}$, we see that they are $\GL(n)$\nobreakdash-de\-pen\-dent if and only if the row spaces of these matrices are $\GL(n)$-dependent. 

In these terms, Theorem~\ref{maintheorem} states the following.

\begin{theorem}
For every $n\in\N$,
any $m+1$ subspaces in $\K^m$ of dimension at most $n$ are $\GL(n)$\nobreakdash-dependent.
\end{theorem}

Observe that $m$ subspaces in $\K^m$ can be $\GL(n)$-independent for every $n$: it suffices to take $L_i$ to be $m$ linearly independent one-dimensional subspaces. 

\medskip
Elementary properties of $n$-linear dependence of subspaces:

\begin{enumerate}
\item If subspaces $L_1,\ldots,L_k$ are $\GL(n)$-dependent, then  $\dim L_i\le n$ for every~$i$.

\item $\GL(1)$-dependence is the ordinary linear dependence of vectors (one-di\-men\-sion\-al linear subspaces).

\item If subspaces $L_1,\ldots,L_k$ are linearly independent, then they are $\GL(n)$\nobreakdash-in\-de\-pen\-dent for every $n$. 

\item Linear dependence of subspaces does not imply even $\GL(1)$-de\-pen\-dence: this implication holds only for 
 one-dimensional subspaces.
\end{enumerate}

\end{document}